\documentclass[a4paper,12pt]{article}
\usepackage{times}
\usepackage[english]{babel}
\usepackage{amssymb,amsmath}
\numberwithin{equation}{section}
\usepackage{amsthm}
\usepackage{array}
\usepackage{wasysym}
\usepackage[noadjust]{cite}
\usepackage{hyperref}
\usepackage[height=22.7cm , width = 16cm , top = 4cm , left = 3cm, a4paper]{geometry}
\usepackage{amssymb}
\usepackage{amsmath}
\usepackage{cite}
 \usepackage{pstricks}
\usepackage{epsfig}
\usepackage{verbatim}
\usepackage{multicol}
\usepackage{graphicx, color, psfrag}
\usepackage{graphics, psfrag}
\usepackage{soul,color}
\usepackage[numbers,sort]{natbib}
 \newtheorem{theorem}{Theorem}[section]

\theoremstyle{definition}

\theoremstyle{ex}
\newtheorem{ex}{Example}[section]
\newcommand{\e}{\end{document}}

\begin{document}

\thispagestyle{empty}

\author{
{{\bf Abdelfattah Mustafa, B. S. El-Desouky and Shamsan AL-Garash}
\newline{\it{{}  }}
 } { }\vspace{.2cm}\\
 \small \it Department of Mathematics, Faculty of Science, Mansoura University, Mansoura 35516, Egypt
}

\title{The Marshall-Olkin Flexible Weibull Extension Distribution}

\date{}

\maketitle
\small \pagestyle{myheadings}
        \markboth{{\scriptsize The Marshall-Olkin Flexible Weibull Extension}}
        {{\scriptsize {Abdelfattah Mustafa, B. S. El-Desouky and Shamsan AL-Garash}}}

\hrule \vskip 8pt

\begin{abstract}
This paper introduces a new generalization of the flexible Weibull distribution with three parameters this model called the Marshall-Olkin flexible Weibull extension (MO-FWE) distribution which exhibits bathtub-shaped hazard rate. We studied it's statistical properties include, quantile function skewness and kurtosis, the mode, $rth$ moments and moment generating function and order statistics. We used the method of maximum likelihood   for estimating the model parameters and the observed Fisher's information matrix is derived. We illustrate the usefulness of the proposed model by applications to real data.
\end{abstract}


\noindent

\section{Introduction}
The Weibull distribution (WD) introduced by Weibull \cite{Weibull1951}, is a popular distribution for modeling lifetime data where the hazard rate function is monotone. Recently appeared new classes of distributions were based on modifications of the Weibull distribution (WD) to provide a good fit to data set with bathtub hazard failure rate Xie and Lai \cite{XieandLai1995}. Among of these, modified Weibull distribution (MWD), Lai et al. \cite{Laietal2003} and Sarhan and Zaindin \cite{SarhanandZaindin2009}, moreover the beta-Weibull distribution (BWD) has been derived by Famoye et al. \cite{Famoyeetal2005}, beta modified Weibull distribution (BMWD), Silva et al. \cite{Silvaetal2010} and Nadarajah et al. \cite{Nadarajahetal2011}. Recently, there are many generalization of the WD like a Kumaraswamy Weibull  distribution (KWD), Cordeiro et al. \cite{Cordeiroetal2010}, generalized modified Weibull  distribution (GMWD), Carrasco et al. \cite{Carrascoetal2008} and exponentiated modified Weibull extension  distribution (EMWED), Sarhan and Apaloo \cite{SarhanandApaloo2013}.

\noindent
The flexible Weibull  distribution (FWED), Bebbington et al. \cite{Bebbingtonetal2007} has a wide range of applications including life testing experiments, reliability analysis, applied statistics and clinical studies. The origin and other aspects of this distribution can be found in \cite{Bebbingtonetal2007}.

\noindent
A random variable $X$ is said to have the flexible Weibull Extension distribution with parameters $\alpha , \beta >0$ if it's probability density function (pdf) is given by

\begin{equation} \label{eq1.1}
g(x)=\left(\alpha +\frac {\beta}{x^2}\right) e^{\alpha x -\frac {\beta }{x}}
\exp\left\{-e^{\alpha x -\frac{\beta }{x}}\right\},  \quad x>0,
\end{equation}

\noindent
while the cumulative distribution function (cdf) is given by

\begin{equation} \label{eq1.2}
G(x)=1-\exp\left\{-e^{\alpha x -\frac{\beta }{x}}\right\},  \quad x>0.
\end{equation}

\noindent
The survival function is given by the equation

\begin{equation} \label{eq1.3}
S(x)=\exp\left\{-e^{\alpha x -\frac{\beta }{x}}\right\},  \quad x>0,
\end{equation}
the hazard rate function is
\begin{equation} \label{eq1.4}
h(x)=\left(\alpha +\frac{\beta }{x^2} \right) e^{\alpha x -\frac{\beta }{x}},
\end{equation}
and the reversed hazard rate function is
\begin{equation} \label{eq1.5}
r(x)=\frac{\left(\alpha +\frac {\beta}{x^2}\right)e^{\alpha x -\frac {\beta }{x}} \exp\left\{-e^{\alpha x -\frac{\beta }{x}}\right\}}{1-\exp\left\{-e^{\alpha x -\frac{\beta }{x}}\right\}}.
\end{equation}

\noindent
In this paper we present a new generalization of the flexible Weibull extension distribution called the Marshall-Olkin flexible Weibull extension  distribution (MO-FWED). By using the Marshall and Olkin's \cite{Marshall1997} method for adding a new parameter to an existing distribution, this new generalized referred to as the Marshall-Olkin flexible Weibull extension distribution.

\noindent
Marshall and Olkin \cite{Marshall1997} proposed a new family of distributions called the Marshall-Olkin extended (MOE) family by adding a new parameter to the baseline distribution. They defined a new survival function $S_{MO}(x)$ by introducing the additional shape parameter $\theta $ such that $\theta >0$ and $\bar{\theta}=1-\theta $. Marshall and Olkin called the parameter $\theta $ the {\em tilt parameter} and they interpreted $\theta $ in terms of the behavior of the hazard rate function of $S_{MO}(x)$. Their ratio is increasing in $x$ for $\theta \geq 1$ and decreasing in $x$ for $0<\theta <1$.
They consider for any arbitrary continuous distribution called {\em baseline distribution} having cumulative distribution function $G(x,\varphi )$ with the related probability density function pdf $g(x,\varphi )$, then the cumulative distribution function of the Marshall Olkin (MO) family of distribution is given by

\begin{equation} \label{eq1.6}
F_{MO}(x)=\frac{G(x,\varphi )}{1-\bar{\theta }S(x,\varphi )}, \quad -\infty <x <\infty ,
\end{equation}

\noindent
where $\theta >0$ and $\bar{\theta}=1-\theta $. \\
The probability density function corresponding to Eq.(\ref{eq1.6}) becomes

\begin{equation} \label{eq1.7}
f_{MO}(x)=\frac{\theta g(x,\varphi )}{\left[1-\bar{\theta }S(x,\varphi )\right]^2}, \quad -\infty <x <\infty.
\end{equation}

\noindent
The survival function, hazard rate function, reversed hazard rate function and cumulative hazard rate function of the Marshall-Olkin (MO) family of a probability distribution are given by

\begin{subequations} \label{eq1.8}
\begin{eqnarray}
S_{MO}(x) &=&  \frac{\theta S(x,\varphi )}{1-\bar{\theta }S(x,\varphi )},
\\
h_{MO}(x) &=&  \frac{h(x,\varphi )}{1-\bar{\theta }S(x,\varphi )},
\\
r_{MO}(x) &=&  \frac{\theta r(x,\varphi )}{1-\bar{\theta }S(x,\varphi )},
\\
H_{MO}(x) &=&  -\log \left( S_{MO}(x)\right )=-\log \left( \frac{\theta S(x,\varphi )}{1-\bar{\theta }S(x,\varphi )}\right ),
\end{eqnarray}
\end{subequations}
respectively, where $\theta >0$, $\bar{\theta}=1-\theta $. \\

\noindent
This paper is organized as follows, we define the cumulative distribution, probability density and hazard functions of the Marshall-Olkin flexible Weibull extension distribution (MO-FWED) in Section 2. In Sections 3 and 4, we introduced the statistical properties include, quantile function, the mode,  skewness and kurtosis, $rth$ moments and moment generating function. The distribution of the order statistics is expressed in Section 5. The maximum likelihood estimation of the parameters is determined in Section 6. Real data sets are analyzed in Section 7 and the results are compared with existing distributions. The conclusions are introduced in Section 8.


\section{Marshall-Olkin Flexible Weibull Extension Distribution}

In this section, we studied the three parameters Marshall-Olkin flexible Weibull extension distribution.
Substituting from  Eqs. (\ref{eq1.2}) and (\ref{eq1.3}) into Eq. (\ref{eq1.6}), the cumulative distribution function of the Marshall-Olkin flexible Weibull extension distribution (MO-FWE) is given by
\begin{equation} \label{eq2.1}
F(x;\alpha ,\beta, \theta  ) =\frac {1-e^{-e^{\alpha x-\frac{\beta }{x}}}}{1-(1-\theta) e^{-e^{\alpha x-\frac{\beta }{x}}}}, \;  x>0, \;  \alpha, \beta, \theta >0.
\end{equation}

\noindent
Substituting from  Eqs. (\ref{eq1.1}) and (\ref{eq1.3}) in Eq. (\ref{eq1.7}), the pdf corresponding to Eq. (\ref{eq2.1}) is given by
\begin{equation} \label{eq2.2}
f(x;\alpha ,\beta,\theta  ) =\frac {\theta \left(\alpha +\frac {\beta}{x^2}\right)e^{\alpha x -\frac {\beta }{x}} e^{-e^{\alpha x -\frac{\beta }{x}}}}{\left [1-(1-\theta ) e^{-e^{\alpha x-\frac{\beta }{x}}}\right ]^2}, \; x>0, \;  \alpha, \beta, \theta >0.
\end{equation}

\noindent
The survival function, hazard rate function, reversed-hazard rate function and cumulative hazard rate function of $X\sim$ MO-FWED($\alpha, \beta, \theta$) are given by
\begin{subequations}\label{eq2.3}
\begin{eqnarray}
S(x;\alpha, \beta, \theta) &=&   \frac{\theta e^{-e^{\alpha x-\frac{\beta }{x}}}}{1-(1-\theta )e^{-e^{\alpha x-\frac{\beta }{x}}}},
\\
h(x;\alpha ,\beta, \theta   ) &=&  \frac{\left(\alpha +\frac {\beta}{x^2}\right) e^{\alpha x -\frac {\beta }{x}}}{1-(1-\theta )e^{-e^{\alpha x-\frac{\beta }{x}}}},
\\
r(x;\alpha ,\beta, \theta  ) &=&  \frac {\left( \alpha +\frac {\beta}{x^2}\right) e^{\alpha x -\frac {\beta }{x}} e^{-e^{\alpha x-\frac{\beta }{x}}}}{\left ( 1- e^{-e^{\alpha x-\frac{\beta }{x}}}\right ) \left [1-(1-\theta )e^{-e^{\alpha x-\frac{\beta }{x}}}\right ]},
\\
H(x;\alpha ,\beta, \theta  ) &=&  -\log \left (\frac{\theta e^{-e^{\alpha x-\frac{\beta }{x}}}}{1-(1-\theta )e^{-e^{\alpha x-\frac{\beta }{x}}}}\right)
\end{eqnarray}
\end{subequations}

\noindent
respectively, $x>0$ and $\alpha, \beta, \theta  >0$.

\noindent
Figures (1--6) display the cdf, pdf, survival function, hazard rate function, reversed hazard rate function and cumulative hazard rate function of the MO-FWED($\alpha$, $\beta$, $\theta  $) for some parameter values.

\begin{center}
\includegraphics[scale=0.7]{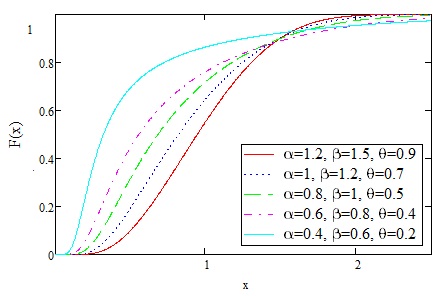}\\
Figure 1: The cdf of MO-FWED for different values of parameters.
\end{center}

\vspace{0.4 cm}

\begin{center}
\includegraphics[scale=0.7]{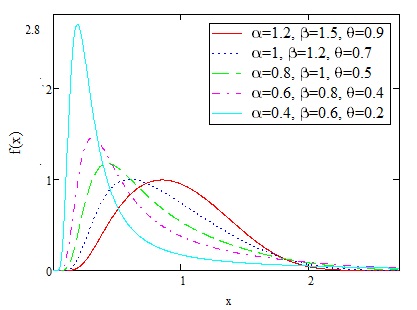}\\
Figure 2: The pdf of MO-FWED for different values of parameters.
\end{center}

\vspace{0.4 cm}

\begin{center}
\includegraphics[scale=0.7]{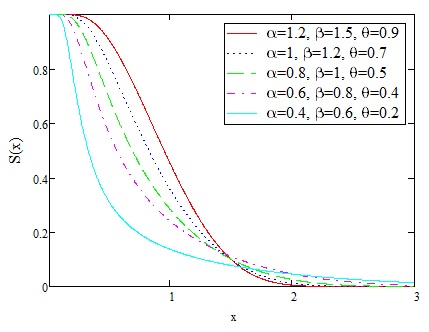}\\
Figure 3: The survival function of MO-FWED for different values of parameters.
\end{center}

\vspace{0.4 cm}

\begin{center}
\includegraphics[scale=0.7]{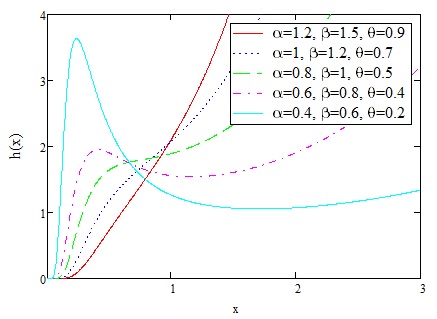}\\
Figure 4: The hazard rate function of MO-FWED for different values of parameters.
\end{center}

\vspace{0.4 cm}

\begin{center}
\includegraphics[scale=0.7]{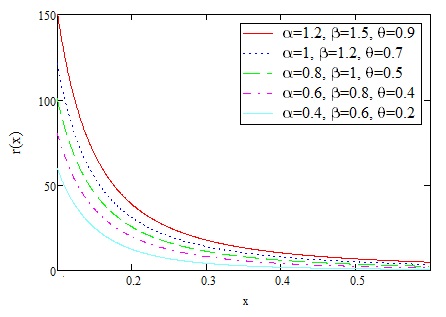}\\
Figure 5: The reversed hazard rate function of MO-FWED  for different values of parameters.
\end{center}

\vspace{0.4 cm}

\begin{center}
\includegraphics[scale=0.7]{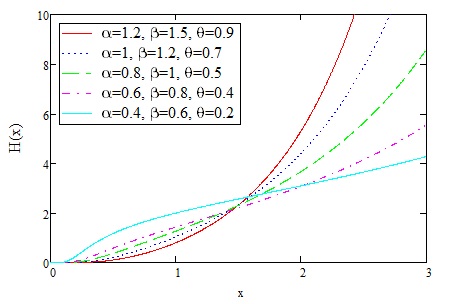}\\
Figure 6: The cumulative hazard rate function of MO-FWED for different values of parameters.
\end{center}

\section{Statistical Properties}

In this section, we study the statistical properties for the MO-FWE distribution, specially quantile and simulation, median, skewness, kurtosis and moments.

\subsection{Quantile and simulation}
The quantile $x_q$ of the MO-FWE($\alpha,\beta, \theta  $)  random variable is given by
\begin{eqnarray} \label{eq3.1}
F(x_q; \alpha,\beta, \theta)&=& q, \quad 0<q<1.
\end{eqnarray}

\noindent
Using the  cumulative distribution function of the MO-FWE distribution, from (\ref{eq2.1}) in Eq. (\ref{eq3.1}), we have
\begin{equation} \label{eq3.2}
\alpha x_q^2-k(q) x_q-\beta=0,
\end{equation}
where
\begin{equation} \label{eq3.3}
k(q)=\ln\left[-\ln \left(\frac{1-q }{1-(1-\theta )q} \right) \right].
\end{equation}

\noindent
So, the simulation of the MO-FWE random variable is straightforward. Let $U$ be a uniform variate on the unit interval $(0,1)$,  thus, by means of the inverse transformation method, we consider the random variable $X$ given by
\begin{equation} \label{eq3.4}
X=\frac{k(u)\pm\sqrt{k(u)^2+4\alpha \beta}}{2\alpha}.
\end{equation}

\noindent
Since the median is $50\%$ quantile then by setting $q=0.5$ in Eq. (\ref{eq3.2}), the median $M$ of the MO-FWED can be obtained the median .

\subsection{The Mode of MO-FWE}
In this subsection, we will derive the mode of the MO-FWED$(\alpha, \beta, \theta)$ by deriving its pdf with respect to $x$ and equal it to zero thus the mode of the MO-FWED$(\alpha , \beta , \theta )$ can be obtained as a nonnegative solution of the following nonlinear equation

\begin{equation} \label{eq3.5}
\left[1-(1-\theta) e^{-e^{\alpha x-\frac{\beta}{x}}} \right] \left[-2\beta x+\left(\alpha x^2+\beta \right)^2 \right]-
\left(\alpha x^2 +\beta \right)^2 \left[1+(1-\theta) e^{-e^{\alpha x-\frac{\beta}{x}}} \right]e^{\alpha x-\frac{\beta}{x}}
=0.
\end{equation}

\noindent
From Figure 2, the pdf for MO-FWED has only one peak, It is a unimodal distribution, so the above equation has only one solution. It is not possible to get an explicit solution of Eq.\ref{eq3.5} in the general case. Numerical methods should be used such as bisection or fixed-point method to solve it.

\subsection{The Skewness and Kurtosis}
The analysis of the variability Skewness and Kurtosis on the shape parameters $\alpha,\beta, \theta $  can be investigated based on quantile measures. The short comings of the classical Kurtosis measure are well-known. The Bowely's skewness based on quartiles is given by, Kenney and Keeping \cite{KenneyandKeeping1962},
\begin{equation} \label{eq3.6}
S_k=\frac{ q_{(0.75)} -2 q_{(0.5)}+q_{(0.25)}}{q_{(0.75)}- q_{(0.25)}},
\end{equation}

\noindent
and the Moors Kurtosis based on quantiles, Moors \cite{Moors1998},
\begin{equation} \label{eq3.7}
K_u=\frac{ q_{(0.875)} - q_{(0.625)}-q_{(0.375)}+ q_{(0.125)}}{q_{(0.75)}- q_{(0.25)}},
\end{equation}

\noindent
where $q_{(.)}$ represents quantile function.

\subsection{The Moments}
Now in this subsection, we derive the $rth$ moment for MO-FWED. Moments are important in any statistical analysis, especially in applications. It can be used to study the most important features and characteristics of  a distribution (e.g. tendency, dispersion, skewness and kurtosis).
\begin{theorem} \label{Th1}
If $X$ has MO-FWED $(\alpha,\beta, \theta)$, then the $r$th moments of random variable $X$, is given by the following
\begin{equation}  \label{eq3.8}
\mu_ r^{'}= \sum_{k=0}^{\infty}\sum_{j=0}^{\infty}\sum_{i=0}^{\infty}
\frac{(-1)^{i+j}(k+1)^{j+1} \beta ^i \theta (1-\theta )^k}{i! j! (j+1)^{r-2i-1} \alpha^{r-i-1} }   \left[\frac{\Gamma(r-i+1)}{\alpha (j+1)^2}+ \beta \Gamma(r-i-1)\right].
\end{equation}
\end{theorem}

\begin{proof}
The $r$th moment of the random variable $X$ with probability density function $f(x)$ is given by
\begin{equation} \label{eq3.9}
\mu_ r^{'}=\int_0^\infty x^r f(x;\alpha ,\beta, \theta ) dx.
\end{equation}

\noindent
Substituting from Eq. (\ref{eq2.2}) into Eq. (\ref{eq3.9}) we get
\begin{eqnarray*}
\mu_ r^{'} &=& \int_0^\infty x^r \theta \left ( \alpha +\frac {\beta}{x^2}\right )e^{\alpha x -\frac {\beta }{x}} e^{-e^{\alpha x -\frac{\beta }{x}}} \left [1-(1-\theta ) e^{-e^{\alpha x-\frac{\beta }{x}}}\right ]^{-2} dx.
\end{eqnarray*}

\noindent
Since $0<(1-\theta ) e^{-e^{\alpha x-\frac{\beta }{x}}}<1$ for $x>0$ we can use the binomial series expansion of\\
$\left [1-(1-\theta ) e^{-e^{\alpha x-\frac{\beta }{x}}}\right ]^{-2}$ yields
\[
\left [1-(1-\theta ) e^{-e^{\alpha x-\frac{\beta }{x}}}\right ]^{-2}=\sum_{k=0}^{\infty} (k+1) (1-\theta )^k e^{-k e^{\alpha x -\frac{\beta }{x}}},
\]
then we get
\begin{equation*}
\mu_ r^{'}= \sum_{k=0}^{\infty} (k+1) \theta (1-\theta )^k \int_0^{\infty} x^r \left(\alpha  +\frac{\beta }{x^2}\right) e^{\alpha x - \frac{\beta }{x}} e^{-(k+1)e^{\alpha x - \frac{\beta }{x}}} dx,
\end{equation*}
\noindent
using series expansion of $e^{-(k+1) e^{\alpha x - \frac{\beta }{x}}}$,
\[
 e^{-(k+1) e^{\alpha x - \frac{\beta }{x}}} =\sum_{j=0}^{\infty}\frac{(-1)^j(k+1)^j}{j!} e^{j(\alpha x -\frac{\beta }{x})},
 \]
we obtain
\begin{equation*}
\mu_ r^{'} = \sum_{k=0}^{\infty}\sum_{j=0}^{\infty} \frac{(-1)^j(k+1)^{j+1} \theta (1-\theta )^k}{j!} \int_0^{\infty} x^r \left(\alpha  +\beta x^{-2}\right) e^{(j+1)\alpha  x} e^{-(j+1)\frac{\beta }{x}} dx,
\end{equation*}

\noindent
using series expansion of  $e^{-(j+1)\frac{\beta }{x}}$,
\[
e^{-(j+1)\frac{\beta }{x}}= \sum_{i=0}^{\infty}\frac{(-1)^i (j+1)^i \beta ^i }{i!} x^{-i},
\]
we obtain
\begin{eqnarray*}
\mu_ r^{'}&=& \sum_{k=0}^{\infty}\sum_{j=0}^{\infty}\sum_{i=0}^{\infty}
\frac{(-1)^{i+j}(k+1)^{j+1} (j+1)^i \beta ^i \theta (1-\theta )^k}{i! j!}  \int_0^{\infty} x^{r-i} \left(\alpha  +\beta x^{-2}\right) e^{(j+1)\alpha x}  dx,
\\
&=& \sum_{k=0}^{\infty}\sum_{j=0}^{\infty}\sum_{i=0}^{\infty} \frac{(-1)^{i+j}(k+1)^{j+1} (j+1)^i \beta ^i \theta (1-\theta )^k}{i! j!} \times
\\
& &
\hspace{2.5cm}
\left[\int_0^{\infty} \alpha x^{r-i} e^{(j+1)\alpha x}  dx  + \int_0^{\infty}\beta x^{r-i-2} e^{(j+1)\alpha x}  dx\right],
\end{eqnarray*}

\noindent
by using the definition of gamma function ( Zwillinger \cite{Zwillinger2014}), in the form,
\[
\Gamma (z)=x^z \int_0^{\infty}t^{z-1} e^{xt} dt, \quad z,x>0.
\]
Finally, we obtain the $r$th moment of MO-FWE distribution in the form
\begin{eqnarray*}
\mu_ r^{'} & = & \sum_{k=0}^{\infty}\sum_{j=0}^{\infty}\sum_{i=0}^{\infty}
\frac{(-1)^{i+j}(k+1)^{j+1} (j+1)^i \beta ^i \theta (1-\theta )^k }{i! j!}  \times
\\
&&
\hspace{2.5cm}
\left[\frac{\Gamma(r-i+1)}{\alpha ^{r-i} (j+1)^{r-i+1}}+\frac{\beta \Gamma(r-i-1)}{\alpha ^{r-i-1} (j+1)^{r-i-1}}\right]
\\
& = & \sum_{k=0}^{\infty}\sum_{j=0}^{\infty}\sum_{i=0}^{\infty}
\frac{(-1)^{i+j}(k+1)^{j+1} \beta ^i \theta (1-\theta )^k }{i! j! (j+1)^{r-2i-1} \alpha^{r-i-1}}  \left[\frac{\Gamma(r-i+1)}{\alpha (j+1)^2}+\beta \Gamma(r-i-1)\right].
\end{eqnarray*}
This completes the proof.

\end{proof}

\section{The Moment Generating Function}
The moment generating function (mgf), $M_X (t)$, of a random variable $X$ provides the basis of an alternative route to analytic results compared with working directly with the pdf and cdf of $X$.
\begin{theorem}
The moment generating function (mgf) of MO-FWE distribution is given by
\begin{equation} \label{eq4.1}
M_X(t) = \sum_{r=0}^\infty \sum_{k=0}^{\infty}\sum_{j=0}^{\infty}\sum_{i=0}^{\infty}
\frac{(-1)^{i+j}(k+1)^{j+1} \beta ^i \theta (1-\theta )^k t^r }{r! i! j! (j+1)^{r-2i-1} \alpha^{r-i-1}}  \left[\frac{\Gamma(r-i+1)}{\alpha (j+1)^2}+\beta \Gamma(r-i-1)\right].
\end{equation}
\end{theorem}
\begin{proof}
The moment generating function of the random variable $X$ with probability density function $f(x)$ is given by
\begin{equation} \label{eq4.2}
M_X(t)=\int_0^\infty e^{tx} f(x) dx,
\end{equation}

\noindent
using series expansion of  $e^{tx}$, we obtain
\begin{equation} \label{eq4.3}
M_X(t)=\sum_{r=0}^{\infty} \frac{t^r}{r!} \int_0^\infty x^r f(x) dx= \sum_{r=0}^{\infty} \frac{t^r}{r!} \mu_ r^{'}.
\end{equation}

\noindent
Substituting from Eq. (\ref{eq3.8}) into Eq. (\ref{eq4.3}), we obtain the moment generating function of MO-FWED in the following form

\begin{equation*}
M_X(t) = \sum_{r=0}^\infty \sum_{k=0}^{\infty}\sum_{j=0}^{\infty}\sum_{i=0}^{\infty}
\frac{(-1)^{i+j}(k+1)^{j+1} \beta ^i \theta (1-\theta )^k t^r }{r! i! j! (j+1)^{r-2i-1} \alpha^{r-i-1}}  \left[\frac{\Gamma(r-i+1)}{\alpha (j+1)^2}+\beta \Gamma(r-i-1)\right].
\end{equation*}

\noindent
This completes the proof.
\end{proof}

\section{Order Statistics }

In this section, we derive closed form expressions for the probability density function of the $r$th order statistic of the MO-FWED. Let $X_{1:n}, X_{2:n}, \cdots, X_{n:n}$ denote the order statistics obtained from a random sample $X_1$, $X_2$, $\cdots$, $X_n$ which taken from a continuous population with cumulative distribution function $F(x;\varphi )$ and probability density function $f(x;\varphi )$, then the probability density function of $X_{r:n}$ is given by

\begin{equation} \label{eq5.1}
f_{r:n}(x;\varphi )=\frac{1}{ B(r, n-r+1)}\left[F(x;\varphi)\right]^{r-1} \left[1-F(x;\varphi )\right]^{n-r} f(x;\varphi),
\end{equation}

\noindent
where $f(x;\varphi)$, $F(x;\varphi)$ are the pdf and cdf of MO-FWED $(\alpha, \beta, \theta )$ given by Eqs. (\ref{eq2.2}) and (\ref{eq2.1}) respectively, $\varphi =(\alpha, \beta, \theta)$ and $B(.,.)$ is the beta function, also we define first order statistics $X_{1:n}= \min(X_1, X_2, \cdots, X_n)$, and the last order statistics as $X_{n:n}= \max (X_1, X_2, \cdots, X_n)$. Since $0 < F(x;\varphi )< 1$  for $x>0$, we can use the binomial expansion of $[1-F(x;\varphi )]^{n-r}$ given as follows.
\begin{equation} \label{eq5.2}
\left[1-F(x;\varphi )\right]^{n-r}=\sum_{i=0}^{n-r}\binom{ n-r}{ i}  (-1)^i [F(x;\varphi )]^i.
\end{equation}

\noindent
Substituting from Eq. (\ref{eq5.2}) into Eq. (\ref{eq5.1}), we obtain
\begin{eqnarray} \label{eq5.3}
f_{r:n}(x;\varphi ) & = & \frac{1}{B(r, n-r+1)} f(x; \varphi )\sum_{i=0}^{n-r} \binom{ n-r}{ i}  (-1)^i
\left[ F(x; \varphi) \right]^{i+r-1}
\nonumber\\
& = & \sum_{i=0}^{n-r} \frac{(-1)^i n!}{i! (r-1)! (n-r-i)!} \left[F(x,\varphi)\right]^{i+r-1} f(x;\varphi).
\end{eqnarray}

\noindent
Substituting from Eqs. (\ref{eq2.1}) and  (\ref{eq3.2}) into Eq. (\ref{eq5.3}), we obtain pdf of $rth$ order statistics for MO-FWED($ \alpha, \beta, \theta$).

\noindent
Relation (\ref{eq5.3}), shows that $f_{r:n}(x;\varphi )$ is the weighted average of the Marshall Olkin flexible Weibull extension MO-FWED withe different shape parameters.

\section{Parameters Estimation}

In this section, point and interval estimation of the unknown parameters of the MO-FWED are derived by using the method of maximum likelihood based on a complete sample.

\subsection{Maximum likelihood estimation}
Let $x_1, x_2,\cdots, x_n$ denote a random sample of complete data from the MO-FWED. The likelihood function is given as
\begin{equation} \label{eq6.1}
L = \prod_{i=1}^{n} f(x_i; \alpha, \beta, \theta),
\end{equation}

\noindent
substituting from Eq. (\ref{eq2.2}) into Eq. (\ref{eq6.1}), we have
\begin{equation*}
L = \prod_{i=1}^{n} \theta \left( \alpha + \frac{\beta}{x_i^2} \right ) e^{\alpha x_i -\frac{\beta }{x_i}}
e^{-e^{\alpha x_i -\frac{\beta }{x_i}}} \left [1-(1-\theta ) e^{-e^{\alpha x_i-\frac{\beta }{x_i}}}\right]^{-2}.
\end{equation*}

\noindent
The log-likelihood function is
\begin{equation} \label{eq6.2}
\mathcal{L} = n \ln(\theta  ) + \sum_{i=1}^{n} \ln \left(\alpha + \frac{\beta }{x_i^2}\right) +
\sum_{i=1}^{n} \left(\alpha x_i -\frac{\beta }{x_i}\right) - \sum_{i=1}^{n} e^{\alpha x_i -\frac{\beta }{x_i}}-2 \sum_{i=1}^{n} \ln\left [ 1-(1-\theta )e^{-e^{\alpha x_i-\frac{\beta }{x_i}}} \right ].
\end{equation}

\noindent
The maximum likelihood estimation of the parameters are obtained by differentiating the log-likelihood function, $\mathcal{L}$, with respect to the parameters $\alpha, \beta$ and $\theta$ and setting the result to zero, we have the following normal equations.
\begin{eqnarray} \label{eq6.3}
\frac{\partial \mathcal{L}}{\partial \alpha } & = &
\sum_{i=1}^{n}\frac{ x_i^2}{\beta+\alpha x_i^2 }+\sum_{i=1}^{n} x_i -\sum_{i=1}^{n} x_i e^{\alpha x_i-\frac{\beta }{x_i}}
+2   \sum_{i=1}^{n} \frac { (1-\theta ) x_i e^{\alpha x_i-\frac{\beta }{x_i}}}{1-\theta -e^{e^{\alpha x_i-\frac{\beta }{x_i}}}} = 0,
\\ \label{eq6.4}
\frac{\partial \mathcal{L}}{\partial \beta  } & = &
\sum_{i=1}^{n}\frac{1}{\beta+\alpha x_i^2}-\sum_{i=1}^{n} \frac{1}{x_i} + \sum_{i=1}^{n} \frac{1}{x_i} e^{\alpha x_i-\frac{\beta }{x_i}} - 2  \sum_{i=1}^{n} \frac{ (1-\theta )e^{\alpha x_i-\frac{\beta }{x_i}}}{x_i \left[ 1-\theta -e^{e^{\alpha x_i-\frac{\beta }{x_i}}} \right]} = 0,
\nonumber\\&&
\\ \label{eq6.5}
\frac{\partial \mathcal{L} }{\partial \theta  } &=&
\frac{n}{\theta }+2\sum_{i=1}^{n} \frac{1}{  1-\theta- e^{e^{\alpha x_i-\frac{\beta }{x_i}}} }= 0.
\end{eqnarray}

\noindent
The MLEs can be obtained by solving the nonlinear equations previous, (\ref{eq6.3})--(\ref{eq6.5}), numerically for $\alpha, \beta$ and $\theta$.

\subsection{Asymptotic confidence bounds}

In this section, we derive the asymptotic confidence intervals when $\alpha, \beta >0$ and $\theta  >0$ as the MLEs of the unknown parameters $\alpha, \beta  >0$ and $\theta >0$ can not be obtained in closed forms, by using variance covariance matrix $I^{-1}$ see Lawless \cite{Lawless2003}, where $I^{-1}$ is the inverse of the observed information matrix which defined as follows.
\begin{equation} \label{eq6.6}
\mathbf{I^{-1}} =
\left(
\begin{array}{ccc}
-\frac{\partial ^2 \mathcal{L}}{\partial \alpha  ^2} & -\frac{\partial ^2 \mathcal{L}}{\partial \alpha  \partial \beta } & -\frac{\partial ^2 \mathcal{L}}{\partial \alpha \partial \theta   }
\\
-\frac{\partial ^2 \mathcal{L}}{\partial \beta  \partial \alpha } & -\frac{\partial ^2 \mathcal{L}}{\partial \beta^2} & -\frac{\partial ^2 \mathcal{L}}{\partial \beta \partial \theta  }
\\
-\frac{\partial ^2 \mathcal{L}}{\partial \theta  \partial \alpha } & -\frac{\partial ^2 \mathcal{L}}{\partial \theta  \partial \beta } & -\frac{\partial ^2 \mathcal{L}}{\partial \theta ^2}
\end{array}
\right)^{-1}
 =
\left(
\begin{array}{ccc}
var(\hat{\alpha }) & cov( \hat{\alpha }, \hat{\beta }) & cov( \hat{\alpha }, \hat{ \theta  })
\\
cov( \hat{\beta },\hat{\alpha  }) & var( \hat{\beta }) & cov( \hat{\beta }, \hat{ \theta  })
\\
cov( \hat{ \theta  }, \hat{\alpha }) & cov( \hat{ \theta }, \hat{\beta }) &  var( \hat{ \theta })
\end{array}
\right).
\end{equation}

\noindent
where
\begin{eqnarray} \label{eq6.7}
\frac{\partial ^2 \mathcal{L}}{\partial \alpha^2 } &= &
-\sum_{i=1}^{n}\frac{x_i^4}{\left(\beta+\alpha x_i^2\right)^2} -\sum_{i=1}^{n} x_i^2 e^{\alpha x_i -\frac{\beta }{x_i}} +
2(1-\theta ) \sum_{i=1}^n x_i^2 \mathcal{B}_i
\\ \label{eq6.8}
\frac{\partial ^2 \mathcal{L}}{\partial \alpha \partial \beta } &= &
-\sum_{i=1}^{n}\frac{x_i^2}{\left(\beta+\alpha x_i^2\right)^2} +\sum_{i=1}^{n} e^{\alpha x_i -\frac{\beta }{x_i}}
-2 (1-\theta ) \sum_{i=1}^n \mathcal{B}_i
\\ \label{eq6.9}
\frac{\partial ^2 \mathcal{L}}{\partial \alpha  \partial \theta   } &= &
 2\sum_{i=1}^{n} x_i \mathcal{A}_i  e^{e^{\alpha x_i-\frac{\beta }{x_i}}}
 \\ \label{eq6.10}
\frac{\partial ^2 \mathcal{L}}{\partial \beta^2 } &= &
-\sum_{i=1}^{n}\frac{1}{\left(\beta+\alpha x_i^2\right)^2} -\sum_{i=1}^{n}\frac{1}{x_i^2} e^{\alpha x_i -\frac{\beta }{x_i}}
+ 2 (1-\theta )\sum_{i=1}^n \frac{ \mathcal{B}_i }{x_i^2}
\\ \label{eq6.11}
\frac{\partial ^2 \mathcal{L}}{\partial \beta  \partial \theta   }&= &
-2\sum_{i=1}^{n}\frac{1}{x_i} \mathcal{A}_i e^{e^{\alpha x_i-\frac{\beta }{x_i}}}
 \\ \label{6.12}
\frac{\partial ^2 \mathcal{L}}{\partial \theta  ^2} &=&
-\frac{n}{\theta^2} - 2\sum_{i=1}^{n}\left[ 1-\theta -e^{e^{\alpha x_i-\frac{\beta }{x_i}}} \right]^{-2}.
\end{eqnarray}
where
\[\mathcal{A}_i=e^{\alpha x_i-\frac{\beta }{x_i}}\left[ 1-\theta - e^{e^{\alpha x_i-\frac{\beta }{x_i}}} \right]^{-2}\text{   and  }
\mathcal{B}_i= \mathcal{A}_i \left[ 1-\theta-e^{e^{\alpha x_i-\frac{\beta }{x_i}}} (1-e^{\alpha x_i -\frac{\beta }{x_i}})\right].
\]

\noindent
We can derive the $(1-\delta)100\%$ confidence intervals of the parameters $\alpha, \beta $ and $\theta $, by using variance matrix as in the following forms
$$ \hat{\alpha} \pm Z_{\frac{\delta}{2}}\sqrt{var(\hat{\alpha })},\quad \hat{\beta} \pm Z_{\frac{\delta}{2}}\sqrt{var(\hat{\beta})}, \quad \hat{\theta  } \pm Z_{\frac{\delta}{2}}\sqrt{var(\hat{\theta  })}, $$
where $Z_{\frac{\delta}{2}}$ is the upper $(\frac{\delta}{2})$-th percentile of the standard normal distribution.

\section{Application}
In this section, we present the analysis of two examples for a real data sets using the MO-FWE $(\alpha, \beta, \theta  )$ model and compare it with the other fitted models like a flexible Weibull extension (FWE), Weibull (W), linear failure rate (LFR), exponentiated Weibull (EW), generalized linear failure rate (GLFR), exponentiated flexible Weibull (EFW), modified Weibull (MW), reduced additive Weibull (RAW) and Extended Weibull (EW) distributions  using Kolmogorov Smirnov (K-S) statistic, as well as Akaike Information Criterion (AIC), \cite{Akaike1974}, Akaike Information Citerion with correction (AICC), Bayesian Information Criterion (BIC) and Hannan-Quinn information criterion (HQIC) \cite{Schwarz1978} values.

\noindent
\begin{ex}
Consider the data have been obtained from Aarset \cite{Aarset1987}, and widely reported in many literatures. It represents the lifetimes of 50 devices, and also, possess a bathtub-shaped failure rate property, Table 1.
\end{ex}

\noindent
\begin{center}
Table 1: Lifetime of 50 devices, see Aarset \cite{Aarset1987}.\\
\begin{tabular}{ccccccccccccccccc} \hline
0.1	& 0.2 & 1  & 1	& 1  & 1  & 1  & 2  & 3  & 6 \\
7   & 11  & 12 & 18	& 18 & 18 & 18 & 18 & 21 & 32 \\	
36  & 40  & 45 & 46 & 47 & 50 & 55 & 60 & 63 & 63 \\	
67  & 67  & 67 & 67 & 72 & 75 & 79 & 82 & 82 & 83 \\	
84  & 84  & 84 & 85 & 85 & 85 & 85 & 85 & 86 & 86 & \\	\hline
\end{tabular}
\end{center}

\noindent
Table 2 gives MLEs of parameters of the MO-FWED  and K--S Statistics. The values of the log-likelihood functions, AIC, AICC, BIC and HQIC are presented in Table 3.

\begin{center}
Table 2: MLEs and K--S of parameters for Aarset data \cite{Aarset1987}.\\
\begin{tabular}{llcc} \hline
Model                    & MLE of the parameters  &  K-S & P-value \\ \hline
Flexible Weibull	 & $\hat{\alpha}$ = 0.0122, $\hat{\beta}$ = 0.7002 & 	0.4386   & 4.29 $\times 10^{-9}$  \\
Weibull              & $\hat{\alpha}$ = 0.0223, $\hat{\beta}$ = 0.949 & 0.2397   & 0.0052 \\
Linear Failure rate  &$\hat{a}$ = 0.014, $\hat{b}$ = 2.4 $\times 10^{-4}$ & 	0.1955   & 0.0370 \\
Exponentiated Weibull & $\hat{\alpha}$ = 0.0109, $\hat{\beta}$ = 4.69, $\hat{\gamma}$ = 0.164 & 0.1841  & 0.0590 \\
Generalized Linear Failure rate & $\hat{a}$ = 0.0038, $\hat{b}$ = 3.04 $\times 10^{-4}$, $\hat{c}$ = 0.533  & 	0.1620   & 0.1293 \\
Exponentiated Flexible Weibull & $\hat{\alpha}$ = 0.0147, $\hat{\beta}$ = 0.133, $\hat{\theta}$ = 4.22  & 0.1433 & 0.2617  \\
MO-FWE($\alpha, \beta, \theta $ ) & $\hat{\alpha}$ = 0.017, $\hat{\beta}$ = 0.401, $\hat{\theta  }$ = 9.043 & 0.1269 & 0.3756 \\ \hline
\end{tabular}
\label{tab1}
\end{center}

\vspace{0.4 cm}
\begin{center}
Table 3: Log-likelihood, AIC, AICC, BIC and HQIC values of models fitted for Aarset data \cite{Aarset1987}.\\
\begin{tabular}{lccccccc} \hline
Model	            & $\mathcal{L}$ &-2 $\mathcal{L}$ &	AIC	& AICC	& BIC & HQIC  \\ \hline
FW($\alpha, \beta $)	            & -250.810 & 501.620  & 505.620  & 505.88	 & 509.448 &507.0762  \\
W($\alpha, \beta$ ) 	& -241.002	& 482.004 & 486.004 & 486.26 & 489.828&487.4602 \\
LFR (a, b)            & -238.064 & 476.128 & 480.128 & 480.38 & 483.952&481.5842 \\
EW($\alpha, \beta, \gamma$ ) 	& -235.926	& 471.852 & 477.852 & 478.37 & 483.588& 480.0363\\
GLFR(a, b, c)  & -233.145 & 466.290	& 472.290 & 472.81	& 478.026 &474.4743\\
EFW($\alpha , \beta , \theta$ )  & -226.989& 453.978	& 459.979 & 460.65	& 465.715&462.1623 \\
MO-FWE($\alpha , \beta , \theta $ ) & -223.755 & 447.510 & 453.510 & 454.03 & 459.246 & 455.6943\\ \hline
\end{tabular}
\label{tab2}
\end{center}

\noindent
We find that the MO-FWE distribution with three parameters provides a better fit than the previous models flexible Weibull (FW), Weibull (W), linear failure rate (LFR), exponentiated Weibull (EW), generalized linear failure rate (GLFR) and exponentiated flexible Weibull (EFW). It has the largest likelihood, and the smallest K-S, AIC, AICC, BIC and HQIC values among those considered in this paper.\\

\noindent
Substituting the MLE's of the unknown parameters $ \alpha, \beta$ and $\theta $  into (\ref{eq6.6}), we get estimation of the variance covariance matrix as the following

$$
I^{-1}=\left(
\begin{array}{rrr}
1.523\times 10^{-6}  & -1.782\times 10^{-5}	& 2.177\times 10^{-3} \\
-1.782\times 10^{-5}	& 0.022    & -0.061	 \\
2.177\times 10^{-3} 	& -0.061	& 8.458	    \\

\end{array}
\right)
$$

\noindent
The approximate 95\% two sided confidence intervals of the unknown parameters $\alpha, \beta  $  and $\theta $ are $\left[ 0.015, 0.019\right]$, $\left[ 0.108, 0.694\right]$ and $\left[ 3.343, 14.743\right]$, respectively.\\

\noindent
To show that the likelihood equation have unique solution, we plot the profiles of the log-likelihood function of $\alpha,  \beta $ and $\theta$ in Figures 7, 8.

\begin{center}
\includegraphics[scale=0.45]{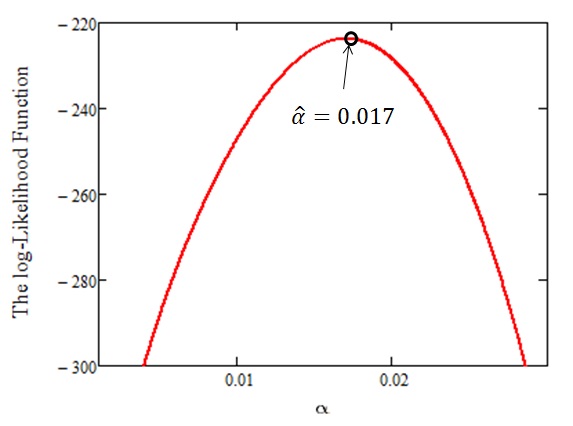}
\includegraphics[scale=0.45]{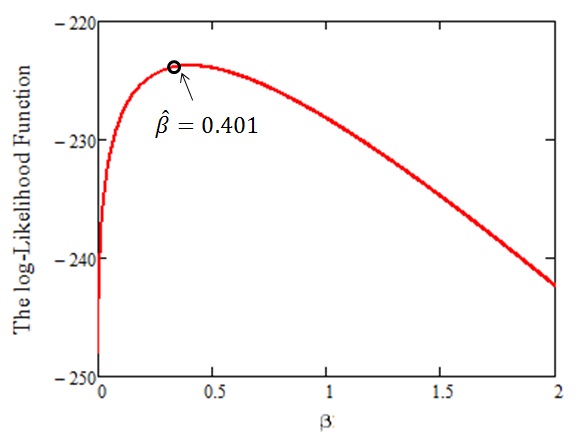}\\
Figure 7: The profile of the log-likelihood function of $\alpha, \beta $ for the Aarset data.
\end{center}

\vspace{0.4 cm}

\begin{center}
\includegraphics[scale=0.55]{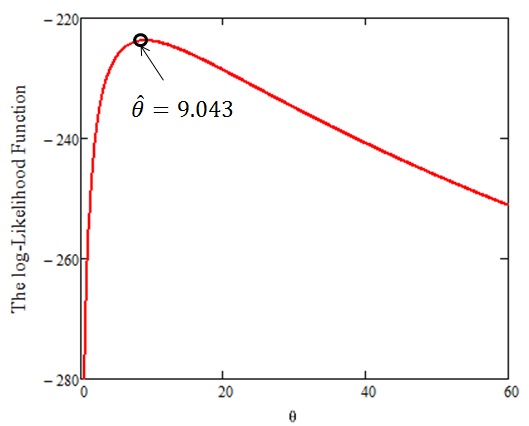}\\
Figure 8: The profile of the log-likelihood function of $ \theta $ for the Aarset data.
\end{center}

\noindent
 The nonparametric estimate of the survival function using the Kaplan-Meier method and its fitted parametric estimations when the distribution is assumed to be MO-FWE, FW, W, LFR, EW, GLFR and EFW are computed and plotted in Figure 9.

\begin{center}
\includegraphics[scale=0.45]{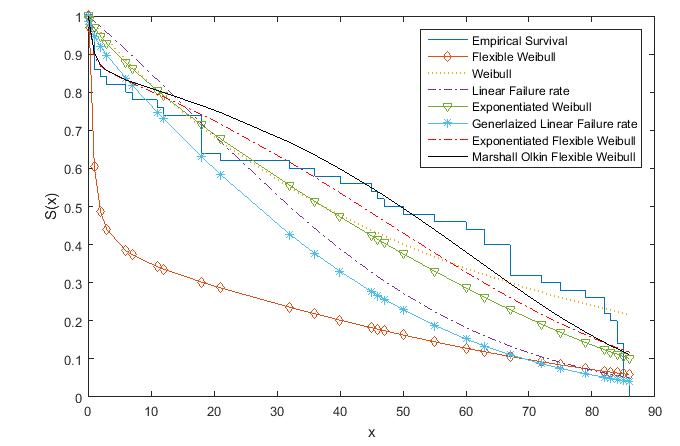}\\
Figure 9: The Kaplan-Meier estimate of the survival function for the Aarset(1987) data.
\end{center}

\noindent
Figures 10 and 11 give the form of the hazard rate and cdf for the MO-FWE, FW, W, LFR, EW, GLFR and EFW  which are used to fit the  Aarset(1987) data after replacing the unknown parameters included in each distribution by their MLE.

\begin{center}
\includegraphics[scale=0.45]{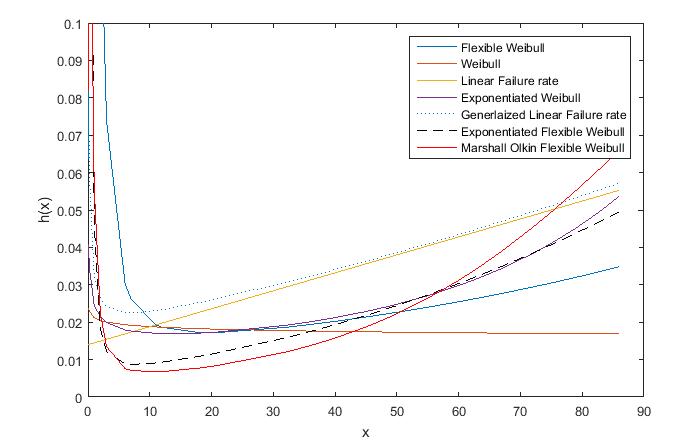}\\
Figure 10: The Fitted hazard rate function for the Aarset(1987) data.
\end{center}

\vspace{0.4 cm}

\begin{center}
\includegraphics[scale=0.45]{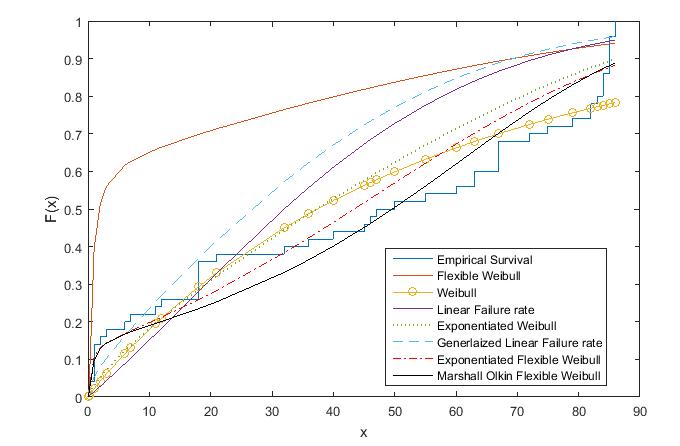}\\
Figure 11: The Fitted cumulative distribution function for the Aarset(1987) data.
\end{center}

\begin{ex}
The data have been obtained from \cite{Salmanetal1999}, it is for the time between failures (thousands of hours) of secondary reactor pumps, Table 4.
\end{ex}

\noindent
\begin{center}
Table 4: Time between failures (thousands of hours) of secondary reactor pumps \cite{Salmanetal1999}\\
\begin{tabular}{ccccccccccccccccc} \hline
2.160	& 0.746   & 0.402   & 0.954	& 0.491 & 6.560	& 4.992  & 0.347 \\
0.150   & 0.358   & 0.101	& 1.359	& 3.465	& 1.060	& 0.614  & 1.921 \\	
4.082   & 0.199	  & 0.605	& 0.273 & 0.070	& 0.062	& 5.320  & \\	\hline
\end{tabular}
\end{center}

\noindent
Table 5 gives MLEs of parameters of the MO-FWE distribution and K-S Statistics. The values of the log-likelihood functions, AIC, AICC, BIC and HQIC are in Table 6.

\begin{center}
Table 5: MLEs and K--S of parameters for secondary reactor pumps.\\
\begin{tabular}{lcccccccc} \hline
Model                    & $\hat{\alpha}$ & $\hat{\beta}$ & $\hat{\theta }$ &  K-S\\ \hline
Flexible Weibull         & 0.0207	      & 2.5875   &  --      & 0.1342\\
Weibull	                 & 0.8077         & 13.9148	 &  --	    & 0.1173\\
Modified Weibull         & 0.1213	      & 0.7924	 &  0.0009	& 0.1188\\
Reduced Additive Weibull & 0.0070	      & 1.7292	 &  0.0452	&  0.1619\\
Extended Weibull 	     & 0.4189	      & 1.0212	 &  10.2778	& 0.1057\\ 
MO-FWE	                 & 0.2160	      & 0.2350	 &	1.2960   & 0.0793 \\ \hline
\end{tabular}
\end{center}

\vspace{0.4 cm}

\begin{center}
Table 6: Log-likelihood, AIC, AICC, BIC and HQIC values of models fitted.\\
\begin{tabular}{lccccccc} \hline
Model	            & $\mathcal{L}$ &-2 $\mathcal{L}$ &	AIC	& AICC	& BIC  & HQIC \\ \hline
Flexible Weibull 	& -83.3424	& 166.6848 & 170.6848 & 171.2848 & 172.95579 &  171.2559\\
Weibull	            & -85.4734	& 170.9468 & 174.9468 & 175.5468 & 177.21779 &  175.5179\\
Modified Weibull 	& -85.4677	& 170.9354 & 176.9354 & 178.1986 & 180.34188 & 	177.7921\\
Reduced Additive Weibull & -86.0728& 172.1456	& 178.1456	    & 179.4088	& 181.55208&  179.0023\\
Extended Weibull 	&  -86.6343 &173.2686  & 179.2686 & 180.5318 & 182.67508 & 	180.1253\\ 
MO-FWE & -30.2110    & 60.4220 & 66.4220  & 67.6852 & 69.8285 & 67.2787 \\ \hline
\end{tabular}
\label{tab4}
\end{center}

\noindent
We find that the MO-FWE distribution with the three-number of parameters provides a better fit than the previous new modified Weibull distributions like a flexible Weibull (FW), Weibull (W), modified Weibull (MW), reduced additive Weibull (RAW) and extended Weibull (EW) distributions. It has the largest likelihood, and the smallest K-S, AIC, AICC, BIC and HQIC values among those considered in this paper.\\

\noindent
Substituting the MLE's of the unknown parameters $ \alpha, \beta $ and $\theta $  into (\ref{eq6.6}), we get estimation of the variance covariance matrix as the following
\[
I_0^{-1}=\left(
\begin{array}{rrr}
1.996\times 10^{-3}  & -7.744\times 10^{-4}	& 8.987\times 10^{-3} \\
-7.744\times 10^{-4}	& 5.487\times 10^{-3}    & -0.022	 \\
8.987\times 10^{-3} 	& -0.022	& 0.326	    \\
\end{array}
\right)
\]

\noindent
The approximate 95\% two sided confidence intervals of the unknown parameters $\alpha, \beta  $  and $\theta $ are $\left[ 0.128, 0.304\right]$, $\left[ 0.09, 0.38\right]$ and $\left[ 0.177, 2.415\right]$, respectively.\\

\noindent
To show that the likelihood equation have unique solution, we plot the profiles of the log-likelihood function of $\alpha,  \beta $ and $\theta$ in Figures 12 and 13.

\begin{center}
\includegraphics[scale=0.45]{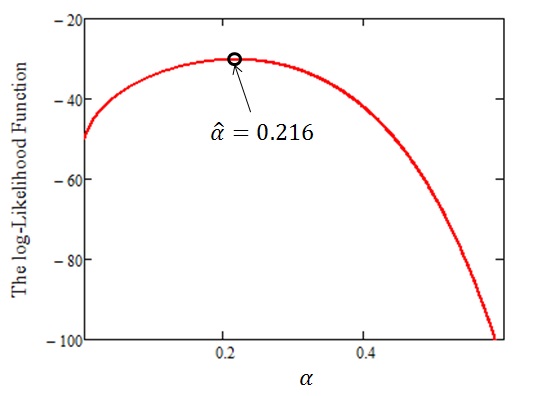}
\includegraphics[scale=0.43]{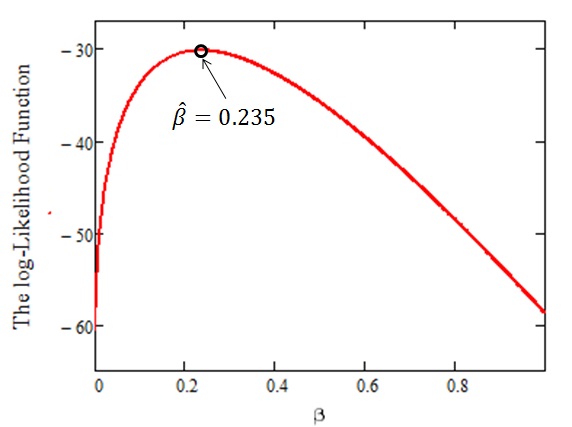}\\
Figure 12: The profile of the log-likelihood function of $\alpha $ and $\beta$.
\end{center}

\vspace{0.4 cm}
\begin{center}
\includegraphics[scale=0.5]{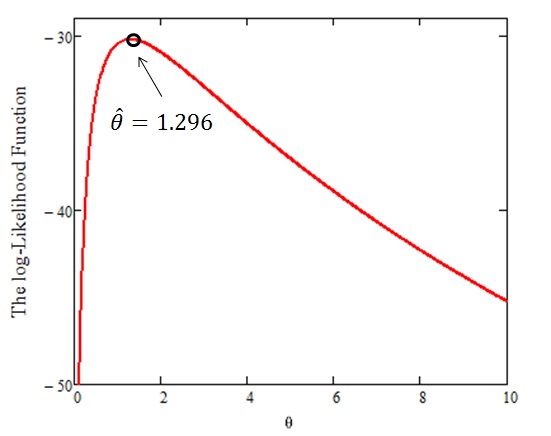}\\
Figure 13: The profile of the log-likelihood function of $\theta$.
\end{center}

\noindent
 The nonparametric estimate of the survival function using the Kaplan-Meier method and its fitted parametric estimations when the distribution is assumed to be MO-FWE, FW, W, MW, RAW and EW are computed and plotted in Figure 14.

\begin{center}
\includegraphics[scale=0.4]{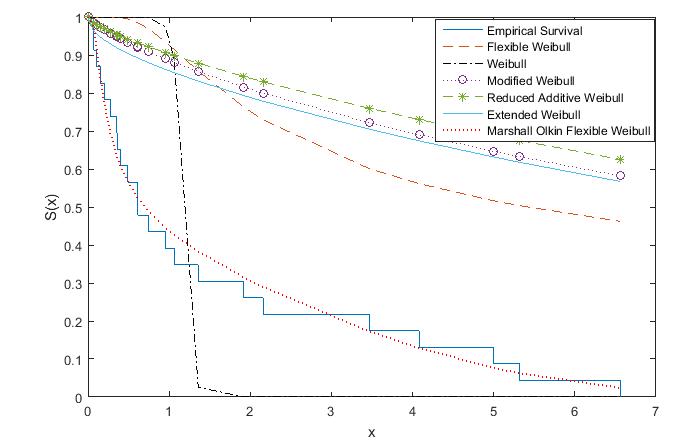}\\
Figure 14: The Kaplan-Meier estimate of the survival function for the data.
\end{center}

\noindent
Figure 15 gives the form of the CDF for the MO-FWE, FW, W, MW, RAW and EW  which are used to fit the data after replacing the unknown parameters included in each distribution by their MLE.

\begin{center}
\includegraphics[scale=0.4]{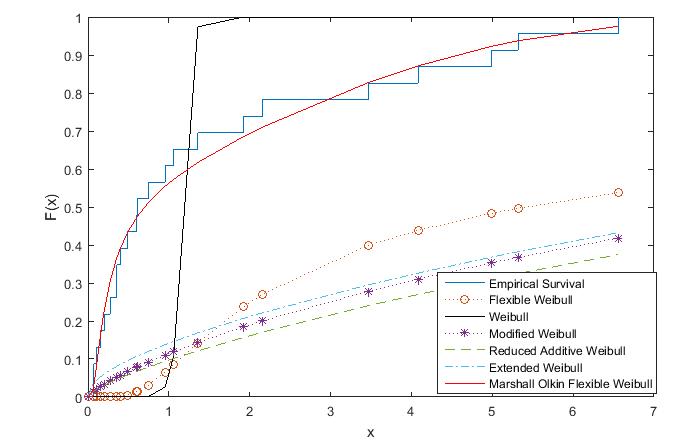}\\
Figure 15: The Fitted cumulative distribution function for the data.
\end{center}

\section{Conclusions}
A new distribution MO-FWE, it's generalized of the flexible Weibull extension distribution based on  the Marshall and Olkin's method, has been proposed and its properties are studied. The idea is to add parameter to a flexible Weibull extension distribution, so that the hazard function is either increasing or more importantly, bathtub shaped. Using  Marshall and Olkin extended family by adding a new parameter to the baseline distribution, the distribution has flexibility to model the second peak in a distribution. We have shown that the Marshall Olkin flexible Weibull extension distribution fits certain well-known data sets better than existing modifications.


\end{document}